\documentclass[11pt]{article}

\usepackage{amscd}
\usepackage{graphics}
\usepackage{color}
\usepackage{cancel}
\usepackage{stmaryrd}
\usepackage{tikz}
\usetikzlibrary{matrix,arrows}
\usepackage{amsmath}
\usepackage{amsthm}
\usepackage{amssymb}
\usepackage{proof}
\usepackage{verbatim,cmll}
\usepackage{setspace}
\usepackage{lscape}
\usepackage{latexsym,relsize}
\usepackage{amscd}
\usepackage{multicol}
\usepackage{bussproofs}
\usepackage{txfonts}
\usepackage{tikz}
\usepackage[position=top]{subfig}
\usepackage{leqno}
\usepackage{authblk}
\usepackage{marvosym}
\usepackage[left=2.5cm,top=3cm,right=2.5cm,bottom=2cm]{geometry}
\DeclareSymbolFont{extraup}{U}{zavm}{m}{n}
\DeclareMathSymbol{\varheart}{\mathalpha}{extraup}{86}
\DeclareMathSymbol{\vardiamond}{\mathalpha}{extraup}{87}
\DeclareMathSymbol{\vardiamond}{\mathalpha}{extraup}{87}

\newcommand{\commment}[1]{}
\renewcommand{\phi}{\varphi}
\renewcommand{\emptyset}{\varnothing}

\renewcommand{\epsilon}{\varepsilon}

\theoremstyle{plain}
\newtheorem{thm}{Theorem}
\newtheorem{theorem}{Theorem}[section]

\newtheorem{lemma}[thm]{Lemma}

\theoremstyle{definition}

\newtheorem{definition}[thm]{Definition}

\title{Sahlqvist Correspondence Theory for Instantial Neighbourhood Logic}
\author{Zhiguang Zhao}
\date{}
\begin{document}
\maketitle
\begin{abstract}
\noindent

In the present paper, we investigate the Sahlqvist-type correspondence theory for instantial neighbourhood logic (INL), which can talk about existential information about the neighbourhoods of a given world and is a mixture between relational semantics and neighbourhood semantics. We have two proofs of the correspondence results, the first proof is obtained by using standard translation and minimal valuation techniques directly, the second proof follows \cite{GeNaVe05} and \cite{Ha03}, where we use bimodal translation method to reduce the correspondence problem in instantial neighbourhood logic to normal bimodal logics in classical Kripke semantics. We give some remarks and future directions at the end of the paper.

{\em Keywords:} instantial neighbourhood logic, modal logic, neighbourhood semantics, Sahlqvist correspondence theory, translation method.

{\em Math. Subject Class.} 03B45, 03B99.

\end{abstract}

\section{Introduction}

Recently \cite{BeEndG20,Tu16,vBBeEn19a,vBBeEn19b,vBBeEnYu17,Yu18,Yu20}, a variant of neighbourhood semantics for modal logics is given, under the name of instantial neighbourhood logic (INL), where existential information about the neighbourhoods of a given world can be added. This semantics is a mixture between relational semantics and neighbourhood semantics, and its expressive power is strictly stronger than neighbourhood semantics. In this semantics, the n+1-ary modality $\Box(\psi_1,\ldots,\psi_n;\phi)$ is true at a world $w$ if and only if there exists a neighbourhood $S\in N(w)$ such that $\phi$ is true everywhere in $S$, and each $\psi_i$ is true at $w_i\in S$ for some $w_i$.

Instantial neighbourhood logic is first introduced in \cite{vBBeEnYu17}, where the authors defines the notion of bisimulation for instantial neighbourhood logic, gives a complete axiomatic system, and determines its precise SAT complexity; in \cite{Tu16}, the canonical rules are defined for instantial neighbourhood logic; in \cite{vBBeEn19a}, the game-theoretic aspects of instantial neighbourhood logic is studied; in \cite{vBBeEn19b}, a propositional dynamic logic IPDL is obtained by combining instantial neighbourhood logic with propositional dynamic logic (PDL), its sound and complete axiomatic system is given as well as its finite model property and decidability; in \cite{BeEndG20}, the duality theory for instantial neighbourhood logic is developed via coalgebraic method; in \cite{Yu18}, a tableau system for instantial neighbourhood logic is given which can be used for mechanical proof and countermodel search; in \cite{Yu20}, a cut-free sequent calculus and a constructive proof of its Lyndon interpolation theorem is given. However, the Sahlqvist-type correspondence theory is still unexplored, which is the theme of this paper. 

In this paper, we define the Sahlqvist formulas in the instantial neighbourhood modal language, and give two different proofs of correspondence results. The first proof is given by standard translation and minimal valuation techniques as in \cite[Section 3.6]{BRV01}, while the second proof uses bimodal translation method in monotone modal logic and neighbourhood semantics \cite{Ha03,KrWo99,Pa85,vB03} to show that every Sahlqvist formula in the instantial neighbourhood modal language can be translated into a bimodal Sahlqvist formula in Kripke semantics, and hence has a first-order correspondent. 

The structure of the paper is as follows: in Section \ref{Sec:Prelim}, we give a brief sketch on the preliminaries of instantial neighbourhood logic, including its syntax and neighbourhood semantics. In Section \ref{Sec:Trans1}, we define the standard translation of instantial neighbourhood logic into a two-sorted first-order language. In Section \ref{Sec:Sahl}, we define Sahlqvist formulas in instantial neighbourhood logic, and prove the Sahlqvist correspondence theorem via standard translation and minimal valuation. In Section \ref{Sec:Trans2}, we discuss the translation of instantial neighbourhood logic into normal bimodal logic, and prove Sahlqvist correspondence theorem via this bimodal translation. We give some remarks and further directions in Section \ref{Sec:Concl}.

\section{Preliminaries on instantial neighbourhood logic}\label{Sec:Prelim}

In this section, we collect some preliminaries on instantial neighbourhood logic, which can be found in \cite{vBBeEnYu17}.
\paragraph{Syntax.}The formulas of instantial neighbourhood logic are defined as follows:
$$\phi::=p \mid \bot \mid \top \mid \neg\phi \mid \phi_1\land\phi_2 \mid \phi_1\lor\phi_2 \mid \Box_{n}(\phi_1, \ldots, \phi_n; \phi)$$
where $p\in\mathbf{Prop}$ is a propositional variable, $\Box_{n}$ is an $n$+1-ary modality for each $n\in\mathbb{N}$. $\rightarrow,\leftrightarrow$ can be defined in the standard way. An occurence of $p$ is said to be \emph{positive} (resp. \emph{negative}) in $\phi$ if $p$ is under the scope of an even (resp. odd) number of negations. A formula $\phi$ is positive (resp. negative) if all occurences of propositional variables in $\phi$ are positive (resp. negative).\label{page:Pos:Neg}

\paragraph{Semantics.}

For the semantics of instantial neighbourhood logic, we use neighbourhood frames to interpret the instantial neighbourhood modality, one and the same neighbourhood function for all the $n$+1-ary modalities for all $n\in\mathbb{N}$.

\begin{definition}(Neighbourhood frames and models)
A \emph{neighbourhood frame} is a pair $\mathbb{F}=(W, N)$ where $W\neq\emptyset$ is the set of worlds, $N:W\rightarrow\mathcal{P}(\mathcal{P}(W))$ is a map called a \emph{neighbourhood function}, where $\mathcal{P}(W)$ is the powerset of $W$. A \emph{valuation} on $W$ is a map $V: \mathbf{Prop} \rightarrow \mathcal{P}(W)$. A triple $\mathbb{M}=(W,N,V)$ is called a \emph{neighbourhood model} or a neighbourhood model based on $(W,N)$ if $(W,N)$ is a neighbourhood frame and $V$ is a valuation on it.
\end{definition}

The semantic clauses for the Boolean part is standard. For the instantial neighbourhood modality $\Box$, the satisfaction relation is defined as follows:
\begin{center}
$\mathbb{M},w\Vdash\Box_{n}(\phi_1, \ldots, \phi_n; \phi)$ iff there is $S\in N(w)$ such that for all $s\in S$ we have $\mathbb{M},s\Vdash\phi$ and for all $i=1,\ldots,n$ there is an $s_i\in S$ such that $\mathbb{M},s_i\Vdash\phi_i$.
\end{center}

\paragraph{Semantic properties of instantial neighbourhood modalities}\label{Page:preserv:join}It is easy to see the following lemma, which states that the $n$+1-ary instantial neighbourhood modality $\Box_{n}$ is monotone in every coordinate, and is completely additive (and hence monotone) in the first $n$ coordinates. This observation is useful in the algebraic correspondence analysis in instantial neighbourhood logic.

\begin{lemma}
\begin{enumerate}
\item For any $\mathbb{F}=(W, N)$, any $w\in W$ and any valuations $V_1,V_2:\mathsf{Prop}\rightarrow\mathcal{P}(W)$ such that $V_1(p)\subseteq V_2(p)$, $V_1(p_i)\subseteq V_2(p_i)$ for all $i=1, \ldots, n$, 
\begin{center}
if $\mathbb{F}, V_1, w\Vdash\Box_{n}(p_1, \ldots, p_n; p)$, then $\mathbb{F}, V_2, w\Vdash\Box_{n}(p_1, \ldots, p_n; p)$;
\end{center}
\item For any $\mathbb{F}=(W, N)$, any $w\in W$ and any valuation $V:\mathsf{Prop}\rightarrow\mathcal{P}(W)$, fix an $i\in \{1,\ldots, n\}$ and a $v\in W$, and define $V_{i,v}:\mathsf{Prop}\rightarrow\mathcal{P}(W)$ such that $V_{i,v}(p_j)=V(p_j)$ for $j\neq i$, and $V_{i,v}(p_i)=\{v\}$. Then the following holds:
\begin{center}
$\mathbb{F}, V, w\Vdash\Box_{n}(p_1, \ldots, p_i,\ldots, p_n; p)$ iff there exists a $v\in V(p_i)$ such that $\mathbb{F}, V_{i,v}, w\Vdash\Box_{n}(p_1, \ldots, p_i,\ldots, p_n; p)$.
\end{center}
\end{enumerate}
\end{lemma}

Algebraically, if we view the $n$+1-ary modality $\Box_{n}$ as an $n$+1-ary function $\Box_{n}^{\mathbb{A}}:\mathbb{A}^{n+1}\rightarrow\mathbb{A}$, then $\Box_{n}^{\mathbb{A}}(a_1, \ldots, a_n; a)$ is completely additive (i.e.\ preserve arbitrary joins) in the first $n$ coordinate, and monotone in the last coordinate. This observation is useful in the algebraic correspondence analysis (see Section \ref{Sec:Concl}).

\section{Standard translation of instantial neighbourhood logic}\label{Sec:Trans1}

\subsection{Two-sorted first-order language $\mathcal{L}_{1}$ and standard translation}\label{Subsec:L1ST}

Given the INL language, we consider the corresponding two sorted first-order language $\mathcal{L}_1$, which is going to be interpreted in a two-sorted domain $W_{w}\times W_{s}$. It has the following ingredients:

\begin{enumerate}
\item world variables $x,y,z,\ldots$, to be interpreted as possible worlds in the world domain $W_{w}$;
\item subset variables $X,Y,Z,\ldots$, to be interpreted as objects in the subset domain $W_{s}=\{X \mid X\subseteq W_{w}\}$;\footnote{Notice that here the subset variables are treated as first-order variables in the subset domain $W_{s}$, rather than second-order variables in the world domain $W_{w}$.}
\item a binary relation symbol $R_{\ni}$, to be interpreted as the reverse membership relation $R^{\ni}\subseteq W_{s}\times W_{w}$ such that $R^{\ni}Xx$ iff $x\in X$;
\item a binary relation symbol $R_{N}$, to be interpreted as the neighbourhood relation $R^{N}\subseteq W_{w}\times W_{s}$ such that $R^{N}xX$ iff $X\in N(x)$;
\item unary predicate symbols $P_{1}$, $P_{2}$,\ldots, to be interpreted as subsets of the world domain $W_{w}$.
\end{enumerate}

We also consider the following second-order language $\mathcal{L}_{2}$ which is obtained by adding second-order quantifiers $\forall P_{1}, \forall P_{2}$,\ldots over the world domain $W_{w}$. Existential second-order quantifiers $\exists P_{1}, \exists P_{2}, \ldots$ are interpreted in the standard way. Notice that here the second-order variables $P_1$,\ldots are different from the subset variables $X,Y,Z,\ldots$, since the former are interpreted as subsets of $W_{w}$, and the latter are interpreted as objects in $W_{s}$.

Now we define the standard translation $ST_{w}(\phi)$ as follows:

\begin{definition}(Standard translation)
For any INL formula $\phi$ and any world symbol $x$, the standard translation $ST_{x}(\phi)$ of $\phi$ at $x$ is defined as follows:
\begin{itemize}
\item $ST_{x}(p):=Px$;
\item $ST_{x}(\bot):=x\neq x$;
\item $ST_{x}(\top):=x= x$;
\item $ST_{x}(\neg\phi):=\neg ST_{x}(\phi)$;
\item $ST_{x}(\phi\land\psi):=ST_{x}(\phi)\land ST_{x}(\psi)$;
\item $ST_{x}(\phi\lor\psi):=ST_{x}(\phi)\lor ST_{x}(\psi)$;
\item $ST_{x}(\Box_{n}(\phi_1, \ldots, \phi_n;\phi))=\exists X(R_{N}xX\land\forall y(R_{\ni}Xy\rightarrow ST_{y}(\phi))\land$

$\exists y_{1}(R_{\ni}Xy_{1}\land ST_{y_{1}}(\phi_{1}))\land\ldots\land\exists y_{n}(R_{\ni}Xy_{n}\land ST_{y_{n}}(\phi_{n})))$.
\end{itemize}
\end{definition}

For any neighbourhood frame $\mathbb{F}=(W,N)$, it is natural to define the following corresponding two-sorted Kripke frame $\mathbb{F}^{2}=(W, \mathcal{P}(W), R^{\ni}, R^{N})$, where 

\begin{enumerate}
\item $R^{\ni}\subseteq \mathcal{P}(W)\times W$ such that for any $x\in W$ and $X\in\mathcal{P}(W)$, $R^{\ni}Xx$ iff $x\in X$;
\item $R^{N}\subseteq W\times\mathcal{P}(W)$ such that for any $x\in W$ and $X\in\mathcal{P}(W)$, $R^{N}xX$ iff $X\in N(x)$.
\end{enumerate}

Given a two-sorted Kripke frame $\mathbb{F}^{2}=(W, \mathcal{P}(W), R^{\ni}, R^{N})$, a valuation $V$ is defined as a map $V:\mathsf{Prop}\rightarrow\mathcal{P}(W)$. Notice that here the $\mathcal{P}(W)$ in the definition of $V$ is understood as the powerset of the first domain, rather than the second domain itself.

For this standard translation, it is easy to see the following correctness result:

\begin{theorem}
For any neighbourhood frame $\mathbb{F}=(W,N)$, any valuation $V$ on $\mathbb{F}$, any $w\in W$, any INL formula $\phi$,
$$(\mathbb{F},V,w)\Vdash\phi\mbox{ iff }\mathbb{F}^{2},V\vDash ST_{x}(\phi)[w].$$
\end{theorem}

\section{Sahlqvist correspondence theorem in instantial neighbourhood logic via standard translation}\label{Sec:Sahl}

In this section, we will define the Sahlqvist formulas in instantial neighbourhood logic and prove the correspondence theorem via standard translation and minimal valuation method. First we recall the definition of Sahlqvist formulas in normal modal logic. Then we identify the special situations where the instantial neighourhood modalities ``behave well'', i.e.\ have good quantifier patterns in the standard translation. Finally, we define INL-Sahlqvist formulas step by step in the style of \cite[Section 3.6]{BRV01}, and prove the correspondence theorem.

\subsection{Sahlqvist formulas in normal modal logic}\label{Subsec:Sahlqvist:basic}
 In this subsection we recall the syntactic definition of Sahlqvist formulas in normal modal logic (see \cite[Section 3.6]{BRV01}).

\begin{definition}(Sahlqvist formulas\footnote{Here what we call Sahlqvist formulas are called Sahlqvist implications in \cite[Section 3.6]{BRV01}.} in normal modal logic) 
A \emph{boxed atom} is a formula of the $\Box_{i_1}\ldots\Box_{i_n}p$, where $\Box_{i_1},\ldots,\Box_{i_n}$ are (not necessarily distinct) boxes. In the case where $k$=0, the boxed atom is just $p$.

A \emph{Sahlqvist antecedent} $\phi$ is a formula built up from $\bot,\top$, boxed atoms, and negative formulas, using $\land,\lor$ and existential modal operators ($\Diamond$ and $\Delta$). A \emph{Sahlqvist formula} is an implication $\phi\to\psi$ in which $\psi$ is positive and $\phi$ is a Sahlqvist antecedent.
\end{definition}

As we can see from the definition above, the Sahlqvist antecedents are built up by $\bot,\top,p,\Box_{i_1}\ldots\Box_{i_n}p$ and negative formulas using $\land,\lor,\Diamond,\Delta$. If we consider the standard translations of Sahlqvist antecedents, the inner part are translated into universal quantifiers, and the outer part are translated into existential quantifiers.

\subsection{Special cases where the instantial neighbourhood modalities become ``normal''}\label{Subsec:normal}

As is mentioned in \cite[Section 7]{vBBeEnYu17} and as we can see in the definition of the standard translation, the quantifier pattern of $\Box_{n}(\phi_1, \ldots, \phi_n;\phi)$ is similar to the case of monotone modal logic \cite{Ha03} which has an $\exists\forall$ pattern. As a result, even with two layers of INL modalities the complexity goes beyond the Sahlqvist fragment. However, we can still consider some special situations where we can reduce the modality to an $n$-ary normal diamond or a unary normal box.

\paragraph{$n$-ary normal diamond.}We first consider the case $\Box_{n}(\phi_1, \ldots, \phi_n;\phi)$ where $\phi$ is a \emph{pure formula} without any propositional variables, i.e., all propositional variables are substituted by $\bot$ or $\top$. In this case $ST_{x}(\phi)$ is a first-order formula $\alpha_{\phi}(x)$ without any unary predicate symbols $P_{1},P_{2}$\ldots. Therefore, in the shape of the standard translation of $\Box_{n}(\phi_1, \ldots, \phi_n;\phi)$, the universal quantifier $\forall y$ is not touched during the computation of minimal valuation, since there is no unary predicate symbol there. Indeed, we can consider the following equivalent form of $ST_{x}(\Box_{n}(\phi_1, \ldots, \phi_n;\phi))$:
$$ST_{x}(\Box_{n}(\phi_1, \ldots, \phi_n;\phi))=\exists X\exists y_{1}\ldots\exists y_{n}(R_{N}xX\land R_{\ni}Xy_{1} \land\ldots\land R_{\ni}Xy_{n}\land\forall y(R_{\ni}Xy\rightarrow \alpha_{\phi}(y))\land$$
$$(ST_{y_{1}}(\phi_{1})\land\ldots\land ST_{y_{n}}(\phi_{n})))$$

Now $ST_{x}(\Box_{n}(\phi_1, \ldots, \phi_n;\phi))$ is essentially in a form similar to $ST_{x}(\Diamond\psi)$ in the normal modal logic case; indeed, when we compute the minimal valuation here, $R_{N}xX\land R_{\ni}Xy_{1} \land\ldots\land R_{\ni}Xy_{n}\land \forall y(R_{\ni}Xy\rightarrow \alpha_{\phi}(y))$ can be recognized as an integrity and stay untouched during the process. 

From now onwards we can denote $\Box_{n}(\phi_1, \ldots, \phi_n;\phi)$ by $\Delta_{n,\phi}(\phi_1, \ldots, \phi_n)$ where $\phi$ is pure.

\paragraph{Unary Box.}As we can see from above, in $\Box_{n}(\phi_1, \ldots, \phi_n;\phi)$, we can replace propositional variables in $\phi$ by $\bot$ and $\top$ to obtain $n$-ary normal diamond modalities. By using the composition with negations, we can get the unary box modality, i.e.\ we can have a modality $$\nabla_{1,\phi}(\phi_1)=\neg\Delta_{1,\phi}(\neg\phi_1)=\neg\Box_{1}(\neg\phi_1;\phi).$$

Now we can consider the standard translation of $\nabla_{1,\phi}(\phi_1)$:
\begin{center}
\begin{tabular}{r c l}
$ST_{x}(\nabla_{1,\phi}(\phi_1))$ & $\leftrightarrow$ & $\neg ST_{x}(\Box_{1}(\neg\phi_1;\phi))$\\
& $\leftrightarrow$ & $\neg\exists X\exists y_{1}(R_{N}xX\land R_{\ni}Xy_{1}\land ST_{y_{1}}(\neg\phi_{1})\land \forall y(R_{\ni}Xy\rightarrow \alpha_{\phi}(y)))$\\
& $\leftrightarrow$ & $\forall X\forall y_{1}\neg(R_{N}xX\land R_{\ni}Xy_{1}\land ST_{y_{1}}(\neg\phi_{1})\land \forall y(R_{\ni}Xy\rightarrow \alpha_{\phi}(y)))$\\
& $\leftrightarrow$ & $\forall X\forall y_{1}(R_{N}xX\land R_{\ni}Xy_{1}\land \forall y(R_{\ni}Xy\rightarrow \alpha_{\phi}(y))\rightarrow ST_{y_{1}}(\phi_{1}))$,\\
\end{tabular}
\end{center}
where $\forall y(R_{\ni}Xy\rightarrow \alpha_{\phi}(y))$ does not contain unary predicate symbols $P_{1},P_{2},\ldots$. Now we can see that 
$ST_{x}(\nabla_{1,\phi}(\phi_1))$ has a form similar to $ST_{x}(\Box\psi)$ where $\Box$ is a normal unary box, by taking $R_{N}xX\land R_{\ni}Xy_{1}\land\forall y(R_{\ni}Xy\rightarrow \alpha_{\phi}(y))$ as an integrity.

\subsection{The definition of INL-Sahlqvist formulas in instantial neighbourhood logic}\label{Subsec:Sahl:Method1}

Now we can define the INL-Sahlqvist formulas in instantial neighbourhood logic step by step in the style of \cite[Section 3.6]{BRV01}.

\subsubsection{Very simple INL-Sahlqvist formulas}\label{SubSub:VSSF}

\begin{definition}[Very simple INL-Sahlqvist formulas]\label{Def:VSSF}
A \emph{very simple INL-Sahlqvist antecedent} $\phi$ is defined as follows:
$$\phi::=p\mid \bot\mid \top\mid  \phi\land\phi \mid \Delta_{n,\theta}(\phi_1, \ldots, \phi_n)\mid \Box_{n}(\phi_1, \ldots, \phi_n;p)$$

where $p\in\mathsf{Prop}$ is a propositional variable, $\theta$ is a pure formula without propositional variables. A \emph{very simple INL-Sahlqvist formula} is an implication $\phi\to\psi$ where $\psi$ is positive (see page \pageref{page:Pos:Neg}), and $\phi$ is a very simple INL-Sahlqvist antecedent.
\end{definition}

For very simple INL-Sahlqvist formulas, we allow $n$-ary normal diamonds $\Delta_{n,\theta}$ in the construction of $\phi$, while for the $n$+1-ary modality $\Box_{n}$, we only allow propositional variables to occur in the $n$+1-th coordinate.

We can show that very simple INL-Sahlqvist formulas have first-order correspondents:

\begin{theorem}\label{Thm:VSSF}
For any given very simple INL-Sahlqvist formula $\phi\to\psi$, there is a two-sorted first-order local correspondent $\alpha(x)$ such that for any neighbourhood frame $\mathbb{F}=(W,N)$, any $w\in W$, 
$$\mathbb{F},w\Vdash \phi\to\psi \mbox{ iff } \mathbb{F}^{2}\vDash\alpha(x)[w].$$
\end{theorem}

\begin{proof}
The proof strategy is similar to \cite[Theorem 3.42, Theorem 3.49]{BRV01}, with some differences in treating $\Box_{n}(\phi_1, \ldots, \phi_n;p)$.

We first start with the two-sorted second-order translation of $\phi\to\psi$, namely $\forall P_1\ \ldots \forall P_n\forall x(ST_x(\phi)\to ST_x(\psi))$, where $ST_x(\phi),ST_x(\psi)$ are the two-sorted first-order standard translations of $\phi,\psi$.

For any very simple INL-Sahlqvist antecedent $\phi$, we consider the shape of $\beta=ST_{x}(\phi)$ defined inductively,

$$\beta::=Px\mid x\neq x\mid x=x\mid \beta\land\beta \mid $$
$$\exists X\exists y_{1}\ldots\exists y_{n}(R_{N}xX\land R_{\ni}Xy_{1} \land\ldots\land R_{\ni}Xy_{n}\land \forall y(R_{\ni}Xy\rightarrow \alpha_{\theta}(y))\land$$
$$ST_{y_{1}}(\phi_{1})\land\ldots\land ST_{y_{n}}(\phi_{n})) \mid $$
$$\exists X\exists y_{1}\ldots\exists y_{n}(R_{N}xX\land R_{\ni}Xy_{1} \land\ldots\land R_{\ni}Xy_{n}\land$$
$$\forall y(R_{\ni}Xy\rightarrow Py)\land ST_{y_{1}}(\phi_{1})\land\ldots\land ST_{y_{n}}(\phi_{n}))$$

Now we can denote $R_{N}xX\land R_{\ni}Xy_{1} \land\ldots\land R_{\ni}Xy_{n}$ as $R_{n}Xxy_{1}\ldots y_{n}$, and thus get 
$$\beta::=Px\mid x\neq x\mid x=x\mid \beta\land\beta \mid $$
$$\exists X\exists y_{1}\ldots\exists y_{n}(R_{n}Xxy_{1}\ldots y_{n}\land \forall y(R_{\ni}Xy\rightarrow \alpha_{\theta}(y))\land ST_{y_{1}}(\phi_{1})\land\ldots\land ST_{y_{n}}(\phi_{n})) \mid $$
$$\exists X\exists y_{1}\ldots\exists y_{n}(R_{n}Xxy_{1}\ldots y_{n}\land \forall y(R_{\ni}Xy\rightarrow Py)\land ST_{y_{1}}(\phi_{1})\land\ldots\land ST_{y_{n}}(\phi_{n}))$$

By using the equivalences
$$\exists y\delta(y)\land\gamma\leftrightarrow\exists y(\delta(y)\land\gamma)\mbox{ (where }y\mbox{ does not occur in }\gamma)$$ 
and 
$$\exists X\delta(X)\land\gamma\leftrightarrow\exists X(\delta(X)\land\gamma)\mbox{ (where }X\mbox{ does not occur in }\gamma),$$

it is easy to see that the two-sorted first-order formula $\beta=ST_{x}(\phi)$ is equivalent to a formula of the form $\exists\overline{X}\exists\overline{y}(\mbox{REL}^{\overline{\theta},\overline{X},x,\overline{y}}\land\mbox{ATProp})$, where:

\begin{itemize}
\item $\mbox{REL}^{\overline{\theta},\overline{X},x,\overline{y}}$ is a (possibly empty) conjunction of formulas of the form $R_{n}Xxy_{1}\ldots y_{n}$ or $\forall y(R_{\ni}Xy\rightarrow \alpha_{\theta}(y))$;
\item $\mbox{ATProp}$ is a conjunction of formulas of the form $\forall y(R_{\ni}Xy\rightarrow Py)$ or $Pw$ or $w=w$ or $w\neq w$.
\end{itemize}

Therefore, by using the equivalences
$$(\exists y\delta(y)\to\gamma)\leftrightarrow\forall y(\delta(y)\to\gamma)\mbox{ (where }y\mbox{ does not occur in }\gamma)$$
and 
$$(\exists X\delta(X)\to\gamma)\leftrightarrow\forall X(\delta(X)\to\gamma)\mbox{ (where }X\mbox{ does not occur in }\gamma),$$
it is immediate that $\forall P_1\ \ldots \forall P_n\forall x(ST_x(\phi)\to ST_x(\psi))$ is equivalent to 
$$\forall P_1\ \ldots \forall P_n\forall\overline{X}\forall x\forall \overline{y}(\mbox{REL}^{\overline{\theta},\overline{X},x,\overline{y}}\land\mbox{ATProp}\to\mbox{POS}), \footnote{Notice that the quantifiers $\forall P_1\ \ldots \forall P_n$ are second-order quantifiers over the world domain, and $\forall\overline{X}$ are first-order quantifiers over the subset domain.}$$where $\mbox{REL}^{\overline{\theta},\overline{X},x,\overline{y}}$ and $\mbox{ATProp}$
are given as above, and $\mbox{POS}$ is the standard translation $ST_x(\psi)$.

Now we can use similar strategy as in \cite[Theorem 3.42, Theorem 3.49]{BRV01}. To make it easier for later parts in the paper, we still mention how the minimal valuation and the resulting first-order correspondent formula look like. Without loss of generality we suppose that for any unary predicate $P$ that occurs in the $\mbox{POS}$ also occurs in $\mbox{AT}$; otherwise we can substitute $P$ by $\lambda u.u\neq u$ for $P$ to eliminate $P$. 

Now consider a unary predicate symbol $P$ occuring in $\mbox{ATProp}$, and $Px_1,\ldots, Px_n$, $\forall y(R_{\ni}X_{1}y\rightarrow Py)$, \ldots, $\forall y(R_{\ni}X_{m}y\rightarrow Py)$ are all occurences of $P$ in $\mbox{ATProp}$. By taking $\sigma(P)$ to be $$\lambda u. u=x_1\lor\ldots\lor u=x_n\lor R_{\ni}X_{1}u\lor\ldots\lor R_{\ni}X_{m}u,$$ we get the minimal valuation. The resulting first-order correspondent formula is 
$$\forall\overline{X}\forall x\forall\overline{y}(\mbox{REL}^{\overline{\theta},\overline{X},x,\overline{y}}\to[\sigma(P_1)/P_1, \ldots, \sigma(P_k)/P_k]\mbox{POS}).$$
\end{proof}

From the proof above, we can see that the part corresponding to $\Delta_{n,\theta}(\phi_1, \ldots, \phi_n)$ is essentially treated in the same way as an $n$-ary diamond in the normal modal logic setting, and $\Box_{n}(\phi_1, \ldots, \phi_n;p)$ is treated as $\Delta(\Diamond\phi_{1}\land\ldots\land\Diamond\phi_{n}\land\Box p)$ where $\Delta$ is an $n$+1-ary normal diamond, $\Diamond$ is a unary normal diamond and $\Box$ is a unary normal box, therefore we can guarantee the compositional structure of quantifiers in the antecedent to be $\exists\forall$ as a whole.

\subsubsection{Simple INL-Sahlqvist formulas}\label{SubSub:SSF}
Similar to simple Sahlqvist formulas in basic modal logic, here we can define simple INL-Sahlqvist formulas:

\begin{definition}[Simple INL-Sahlqvist formulas]\label{Def:SSF}
A \emph{pseudo-boxed atom} $\zeta$ is defined as follows: 
$$\zeta::=p \mid \bot \mid \top \mid \zeta\land\zeta \mid \nabla_{1,\theta}(\zeta)\label{page:pseu:boxed:atom}$$
where $\theta$ is a pure formula without propositional variables. Based on this, a \emph{simple INL-Sahlqvist antecedent} $\phi$ is defined as follows:
$$\phi::=\zeta\mid \bot\mid \top\mid  \phi\land\phi \mid \Delta_{n,\theta}(\phi_1, \ldots, \phi_n)\mid \Box_{n}(\phi_1, \ldots, \phi_n;\zeta)$$

where $\theta$ is a pure formula without propositional variables and $\zeta$ is a pseudo-boxed atom. A \emph{simple INL-Sahlqvist formula} is an implication $\phi\to\psi$ where $\psi$ is positive, and $\phi$ is a simple INL-Sahlqvist antecedent.
\end{definition}

\begin{theorem}\label{Thm:SSF}
For any given simple INL-Sahlqvist formula $\phi\to\psi$, there is a two-sorted first-order local correspondent $\alpha(x)$ such that for any neighbourhood frame $\mathbb{F}=(W,N)$, any $w\in W$, 
$$\mathbb{F},w\Vdash \phi\to\psi \mbox{ iff } \mathbb{F}^{2}\vDash\alpha(x)[w].$$
\end{theorem}

\begin{proof}
We use similar proof strategy as \cite[Theorem 3.49]{BRV01}. The part that we needs to take care of is the way to compute the minimal valuation. Now without loss of generality (by renaming quantified variables) we have the following Backus-Naur form of $\beta=ST_{x}(\zeta)$ defined inductively for any pseudo-boxed atom $\zeta$:
$$\beta::=Px\mid x\neq x\mid x=x\mid \beta\land\beta \mid \forall X\forall y_{1}(R_{N}xX\land R_{\ni}Xy_{1}\land \forall y(R_{\ni}Xy\rightarrow \alpha_{\theta}(y))\rightarrow ST_{y_{1}}(\zeta)).$$

The Backus-Naur form of $\beta=ST_{x}(\phi)$ is defined inductively for any simple Sahlqvist antecedent $\phi$:
$$\beta::=ST_{x}(\zeta)\mid x\neq x\mid x=x\mid \beta\land\beta \mid$$
$$\exists X\exists y_{1}\ldots\exists y_{n}(R_{N}xX\land R_{\ni}Xy_{1} \land\ldots\land R_{\ni}Xy_{n}\land \forall y(R_{\ni}Xy\rightarrow \alpha_{\theta}(y))\land ST_{y_{1}}(\phi_{1})\land\ldots\land ST_{y_{n}}(\phi_{n})) \mid $$
$$\exists X\exists y_{1}\ldots\exists y_{n}(R_{N}xX\land R_{\ni}Xy_{1} \land\ldots\land R_{\ni}Xy_{n}\land \forall y(R_{\ni}Xy\rightarrow ST_{y}(\zeta))\land ST_{y_{1}}(\phi_{1})\land\ldots\land ST_{y_{n}}(\phi_{n}))$$

Now we can denote $R_{N}xX\land R_{\ni}Xy_{1} \land\ldots\land R_{\ni}Xy_{n}$ as $R_{n}Xxy_{1}\ldots y_{n}$ and $R_{-1,\theta}X$ for $\forall y(R_{\ni}Xy\rightarrow \alpha_{\theta}(y))$ (note that the only possible free variable in $\alpha_{\theta}(y)$ is $y$), then the Backus-Naur form of $\beta=ST_{x}(\zeta)$ and $\beta=ST_{x}(\phi)$ can be given as follows:

$$\beta::=Px\mid x\neq x\mid x=x\mid \beta\land\beta \mid \forall y_{1}(\exists X(R_{1}Xxy_{1}\land R_{-1,\theta}X)\rightarrow ST_{y_{1}}(\zeta)),$$

$$\beta::=ST_{x}(\zeta)\mid x\neq x\mid x=x\mid \beta\land\beta \mid$$
$$\exists X\exists y_{1}\ldots\exists y_{n}(R_{n}Xxy_{1}\ldots y_{n}\land R_{-1,\theta}X\land ST_{y_{1}}(\phi_{1})\land\ldots\land ST_{y_{n}}(\phi_{n})) \mid $$
$$\exists X\exists y_{1}\ldots\exists y_{n}(R_{n}Xxy_{1}\ldots y_{n}\land \forall y(R_{\ni}Xy\rightarrow ST_{y}(\zeta))\land ST_{y_{1}}(\phi_{1})\land\ldots\land ST_{y_{n}}(\phi_{n})).$$

Now we denote $\exists X(R_{1}Xxy_{1}\land R_{-1,\theta}X)$ as $R_{-2,\theta}xy_{1}$, and we get the Backus-Naur form of pseudo-boxed atom $\beta=ST_{x}(\zeta)$ as follows:

$$\beta::=Px\mid x\neq x\mid x=x\mid \beta\land\beta \mid \forall y_{1}(R_{-2,\theta}xy_{1}\rightarrow ST_{y_{1}}(\zeta)),$$

Now using the following equivalences:

\begin{itemize}
\item $(\phi\to\forall z(\psi(z)\to\gamma))\leftrightarrow\forall z(\phi\land\psi(z)\to\gamma)$ (where $z$ does not occur in $\phi$);
\item $(\phi\to(\psi\to\gamma))\leftrightarrow(\phi\land\psi\to\gamma)$;
\item $(\phi\to(\psi\land\gamma))\leftrightarrow((\phi\to\psi)\land(\phi\to\gamma))$;
\item $\forall z(\psi (z)\land\gamma(z))\leftrightarrow(\forall z\psi (z)\land\forall z\gamma(z))$;
\end{itemize}

for any pseudo-boxed atom $\zeta$, the first-order formula $ST_{x}(\zeta)$ is equivalent to a conjunction of two-sorted first-order formulas of the form $\forall \overline{y}(\mbox{REL}^{\overline{\theta},x,\overline{y}}\to\mbox{AT})$, where:

\begin{itemize}
\item $\mbox{REL}^{\overline{\theta},x,\overline{y}}$ is a (possibly empty) conjunction of formulas of the form $R_{-2,\theta}yz$;
\item $\mbox{AT}$ is a formula of the form $Pw$ or $w=w$ or $w\neq w$ where $w$ is bounded by $\forall\overline{y}$ (here we do not need to take the conjunction because of $\forall z(\psi (z)\land\gamma(z))\leftrightarrow(\forall z\psi (z)\land\forall z\gamma(z))$).
\end{itemize}

It is easy to see that $\mbox{REL}^{\overline{\theta},x,\overline{y}}$ does not contain any unary predicate symbol $P_{i}$. By the equivalence $(\exists x\phi(x)\to\psi)\leftrightarrow\forall x(\phi(x)\to\psi)$ where $\psi$ does not contain $x$, we can transform $\forall \overline{y}(\mbox{REL}^{\overline{\theta},x,\overline{y}}\to\mbox{AT})$ into $\forall y(\exists\overline{y}'\mbox{REL}^{\overline{\theta},x,\overline{y}}\to\mbox{AT}(y))$, where $\mbox{AT}(y)$ is $Py$ or $y=y$ or $y\neq y$.

We can introduce a new binary relation symbol $R_{\overline{\theta}}xy$ which is $\exists\overline{y}'\mbox{REL}^{\overline{\theta},x,\overline{y}}$. Then $\beta=ST_{x}(\zeta)$ is a conjunction of formulas of the form $\forall y(R_{\overline{\theta}}xy\to\mbox{AT}(y))$.

Now we somehow come back to the situation of the basic normal modal logic case, where $R_{\overline{\theta}}$ is a real relation symbol. The Backus-Naur form of $\beta=ST_{x}(\phi)$ for simple INL-Sahlqvist antecedent $\phi$ can be recursively defined as follows:

$$\beta::=\forall y(R_{\overline{\theta}}xy\to\mbox{AT}(y)) \mid x\neq x\mid x=x\mid Px \mid \beta\land\beta \mid $$
$$\exists X\exists y_{1}\ldots\exists y_{n}(R_{n}Xxy_{1}\ldots y_{n}\land R_{-1,\theta}X \land ST_{y_{1}}(\phi_{1})\land\ldots\land ST_{y_{n}}(\phi_{n})) \mid $$
$$\exists X\exists y_{1}\ldots\exists y_{n}(R_{n}Xxy_{1}\ldots y_{n}\land\forall y(R_{\ni}Xy\rightarrow ST_{y}(\zeta))\land ST_{y_{1}}(\phi_{1})\land\ldots\land ST_{y_{n}}(\phi_{n}))$$

since $ST_{y}(\zeta)$ is a conjunction of formulas of the form $\forall z(R_{\overline{\theta}}yz\to\mbox{AT}(z))$, we have

\begin{center}
\begin{tabular}{r c l}
$\forall y(R_{\ni}Xy\rightarrow ST_{y}(\zeta))$ & $\leftrightarrow$ & $\forall y(R_{\ni}Xy\rightarrow\bigwedge_{i}\forall z_{i}(R_{\overline{\theta}_{i}}yz_{i}\to\mbox{AT}(z_{i})))$\\
 & $\leftrightarrow$ & $\bigwedge_{i}\forall y(R_{\ni}Xy\rightarrow\forall z_{i}(R_{\overline{\theta}_{i}}yz_{i}\to\mbox{AT}(z_{i})))$\\
 & $\leftrightarrow$ & $\bigwedge_{i}\forall y\forall z_{i}(R_{\ni}Xy\rightarrow(R_{\overline{\theta}_{i}}yz_{i}\to\mbox{AT}(z_{i})))$\\
 & $\leftrightarrow$ & $\bigwedge_{i}\forall z_{i}(\exists y(R_{\ni}Xy\land R_{\overline{\theta}_{i}}yz_{i})\to\mbox{AT}(z_{i})))$.\\
\end{tabular}
\end{center}

Now the situation is similar to the very simple INL-Sahlqvist formula case. We can see how the minimal valuation is computed:

\begin{itemize}
\item for the $\forall y(R_{\overline{\theta}}xy\to\mbox{AT}(y))$ part, when $\mbox{AT}(y)$ is $Py$, its corresponding minimal valuation is $\lambda u.R_{\overline{\theta}}xu$; when $\mbox{AT}(y)$ is $y=y$ or $y\neq y$, we can replace $\mbox{AT}(y)$ by $\top$ or $\bot$, respectively;
\item for the $x\neq x$ part, it is equivalent to $\bot$;
\item for the $x=x$ part,  it is equivalent to $\top$;
\item for the $Px$ part, its corresponding minimal valuation is $\lambda u.x=u$;
\item for the $\forall z_{i}(\exists y(R_{\ni}Xy\land R_{\overline{\theta}_{i}}yz_{i})\to\mbox{AT}(z_{i}))$ part, when $\mbox{AT}(z_{i})$ is $Pz_{i}$, its corresponding minimal valuation is $\lambda u.\exists y(R_{\ni}Xy\land R_{\overline{\theta}_{i}}yu)$; when $\mbox{AT}(y)$ is $y=y$ or $y\neq y$, we can replace $\mbox{AT}(y)$ by $\top$ or $\bot$, respectively.
\end{itemize}

Now for each propositional variable $p_i$, we take the minimal valuation to be the disjunction of all the corresponding minimal valuations where the branch has an occurence of $P_i$. By essentially the same argument as in \cite[Theorem 3.49]{BRV01}, we get the first-order correspondent of $\phi\to\psi$.
\end{proof}

\subsubsection{INL-Sahlqvist formulas}\label{SubSub:SF}
In the present section, we add negated formulas and disjunctions in the antecedent part, which is analogous to \cite[Definition 3.51]{BRV01}.

\begin{definition}[INL-Sahlqvist formulas]\label{Def:SF}
An \emph{INL-Sahlqvist antecedent} $\phi$ is defined as follows:
$$\phi::=\zeta\mid \gamma\mid\bot\mid \top\mid \phi\land\phi \mid\phi\lor\phi \mid \Delta_{n,\theta}(\phi_1, \ldots, \phi_n)\mid \Box_{n}(\phi_1, \ldots, \phi_n;\zeta)\mid \Box_{n}(\phi_1, \ldots, \phi_n;\gamma)$$
where $\theta$ is a pure formula without propositional variables, $\zeta$ is a pseudo-boxed atom defined on page \pageref{page:pseu:boxed:atom} and $\gamma$ is a negative formula defined on page \pageref{page:Pos:Neg}. An \emph{INL-Sahlqvist formula} is an implication $\phi\to\psi$ where $\psi$ is positive, and $\phi$ is an INL-Sahlqvist antecedent.
\end{definition}

\begin{theorem}\label{Thm:SF}
For any given INL-Sahlqvist formula $\phi\to\psi$, there is a two-sorted first-order local correspondent $\alpha(x)$ such that for any neighbourhood frame $\mathbb{F}=(W,N)$, any $w\in W$, 
$$\mathbb{F},w\Vdash \phi\to\psi \mbox{ iff } \mathbb{F}^{2}\vDash\alpha(x)[w].$$
\end{theorem}

\begin{proof}
We use similar proof strategy as \cite[Theorem 3.54]{BRV01}. The part that we needs to take care of is the way to compute the minimal valuation. Now for each INL-Sahlqvist antecedent $\phi$, we consider the Backus-Naur form of $\beta=ST_{x}(\phi)$:

$$\beta::=ST_{x}(\zeta) \mid ST_{x}(\gamma) \mid x\neq x\mid x=x\mid \beta\land\beta \mid \beta\lor\beta \mid $$
$$\exists X\exists y_{1}\ldots\exists y_{n}(R_{N}xX\land R_{\ni}Xy_{1} \land\ldots\land R_{\ni}Xy_{n}\land \forall y(R_{\ni}Xy\rightarrow \alpha_{\theta}(y))\land ST_{y_{1}}(\phi_{1})\land\ldots\land ST_{y_{n}}(\phi_{n})) \mid $$
$$\exists X\exists y_{1}\ldots\exists y_{n}(R_{N}xX\land R_{\ni}Xy_{1} \land\ldots\land R_{\ni}Xy_{n}\land \forall y(R_{\ni}Xy\rightarrow ST_{y}(\zeta))\land ST_{y_{1}}(\phi_{1})\land\ldots\land ST_{y_{n}}(\phi_{n}))\mid$$
$$\exists X\exists y_{1}\ldots\exists y_{n}(R_{N}xX\land R_{\ni}Xy_{1} \land\ldots\land R_{\ni}Xy_{n}\land \forall y(R_{\ni}Xy\rightarrow ST_{y}(\gamma))\land ST_{y_{1}}(\phi_{1})\land\ldots\land ST_{y_{n}}(\phi_{n}))$$

where $\theta$ is a pure formula without propositional variables, $\zeta$ is a pseudo-boxed atom defined on page \pageref{page:pseu:boxed:atom} and $\gamma$ is a negative formula defined on page \pageref{page:Pos:Neg}. 

By denoting $R_{N}xX\land R_{\ni}Xy_{1} \land\ldots\land R_{\ni}Xy_{n}$ as $R_{n}Xxy_1,\dots y_n$, $\forall y(R_{\ni}Xy\rightarrow \alpha_{\theta}(y	))$ as $R_{-1,\theta}X$, we can rewrite the Backus-Naur form of $\beta=ST_{x}(\phi)$ as follows:

$$\beta::=ST_{x}(\zeta) \mid ST_{x}(\gamma) \mid x\neq x\mid x=x\mid \beta\land\beta \mid \beta\lor\beta \mid $$
$$\exists X\exists y_{1}\ldots\exists y_{n}(R_{n}Xxy_1,\dots y_n\land R_{-1,\theta}X\land ST_{y_{1}}(\phi_{1})\land\ldots\land ST_{y_{n}}(\phi_{n})) \mid $$
$$\exists X\exists y_{1}\ldots\exists y_{n}(R_{n}Xxy_1,\dots y_n\land \forall y(R_{\ni}Xy\rightarrow ST_{y}(\zeta))\land ST_{y_{1}}(\phi_{1})\land\ldots\land ST_{y_{n}}(\phi_{n}))\mid$$
$$\exists X\exists y_{1}\ldots\exists y_{n}(R_{n}Xxy_1,\dots y_n\land \forall y(R_{\ni}Xy\rightarrow ST_{y}(\gamma))\land ST_{y_{1}}(\phi_{1})\land\ldots\land ST_{y_{n}}(\phi_{n}))$$

where $\theta$ is a pure formula without propositional variables, $\zeta$ is a pseudo-boxed atom defined on page \pageref{page:pseu:boxed:atom} and $\gamma$ is a negative formula defined on page \pageref{page:Pos:Neg}. 

Using the equivalence $\exists y\delta(y)\land\gamma\leftrightarrow\exists y(\delta(y)\land\gamma)$ (where $y$ does not occur in $\gamma$), $\exists y(\alpha\lor\beta)\leftrightarrow\exists y\alpha\lor\exists y\beta$, $(\alpha\lor\beta)\land\gamma\leftrightarrow(\alpha\land\gamma)\lor(\beta\land\gamma)$, it is easy to see that the first-order formula $\beta=ST^{E}_{x}(\phi)$ is equivalent to a formula of the form 
$\bigvee_{i}\exists\overline{X}_{i}\exists\overline{y}_{i}(\mbox{REL}_{i}^{\overline{X}_{i},x,\overline{y}_{i}}\land\mbox{PS-BOXED-AT}_{i}\land\mbox{NEG}_{i})$, where:

\begin{itemize}
\item $\mbox{REL}_{i}^{\overline{X}_{i},x,\overline{y}_{i}}$ is a (possibly empty) conjunction of formulas of the form $R_{n}Xxy_1,\dots y_n$ and $R_{-1,\theta}X$;
\item $\mbox{PS-BOXED-AT}_{i}$ is a conjunction of formulas of the form $ST_{y}(\zeta)$ and $\forall y(R_{\ni}Xy\rightarrow ST_{y}(\zeta))$ where $\zeta$ is a pseudo-boxed atom;
\item $\mbox{NEG}_{i}$ is a conjunction of formulas of the form $ST_{y}(\gamma)$ and $\forall y(R_{\ni}Xy\rightarrow ST_{y}(\gamma))$ where $\gamma$ is a negative formula.
\end{itemize}

Now let us consider the standard translation of INL-Sahlqvist formula $\phi\to\psi$ where $\phi$ is an INL-Sahlqvist antecedent and $\psi$ is a positive formula. For $\beta=ST^{E}_{x}(\phi\to\psi)$, we have the following equivalence:
\begin{center}
\begin{tabular}{c l}
& $\bigvee_{i}\exists\overline{X}_{i}\exists\overline{y}_{i}(\mbox{REL}_{i}^{\overline{X}_{i},x,\overline{y}_{i}}\land\mbox{PS-BOXED-AT}_{i}\land\mbox{NEG}_{i})\to ST_{x}(\psi)$\\
$\Leftrightarrow$\ & $\bigwedge_{i}(\exists\overline{X}_{i}\exists\overline{y}_{i}(\mbox{REL}_{i}^{\overline{X}_{i},x,\overline{y}_{i}}\land\mbox{PS-BOXED-AT}_{i}\land\mbox{NEG}_{i})\to ST_{x}(\psi))$\\
$\Leftrightarrow$\ & $\bigwedge_{i}\forall\overline{X}_{i}\forall\overline{y}_{i}(\mbox{REL}_{i}^{\overline{X}_{i},x,\overline{y}_{i}}\land\mbox{PS-BOXED-AT}_{i}\land\mbox{NEG}_{i}\to ST_{x}(\psi))$\\
$\Leftrightarrow$\ & $\bigwedge_{i}\forall\overline{X}_{i}\forall\overline{y}_{i}(\mbox{REL}_{i}^{\overline{X}_{i},x,\overline{y}_{i}}\land\mbox{PS-BOXED-AT}_{i}\to \neg\mbox{NEG}_{i}\lor ST_{x}(\psi))$\\
\end{tabular}
\end{center}
Now it is easy to see that $\neg\mbox{NEG}_{i}\lor ST_{x}(\psi)$ is equivalent to a first-order formula which is positive in all unary predicates. We can now use essentially the same proof strategy as Theorem \ref{Thm:SSF}.
\end{proof}

As we can see from the proofs above, the key point is the quantifier pattern of the two-sorted standard translation of the modalities, i.e.\ the outer part of the structure of an INL-Sahlqvist antecedent are translated into existential quantifiers, and the inner part are translated into universal quantifiers.

\section{Bimodal translation of instantial neighbourhood logic}\label{Sec:Trans2}

In the present section we give the second proof of Sahlqvist correspondence theorem, by using a bimodal translation into a normal bimodal language. The methodology is similar to \cite{Ha03}, but with slight differences.

\subsection{Normal bimodal language and two-sorted Kripke frame}\label{Subsec:Bimodal}

As we can see in Section \ref{Sec:Trans1}, for any given neighbourhood frame $\mathbb{F}=(W,N)$, there is an associated two-sorted Kripke frame $\mathbb{F}^{2}=(W, \mathcal{P}(W), R^{\ni}, R^{N})$, where 

\begin{enumerate}
\item $R^{\ni}\subseteq \mathcal{P}(W)\times W$ such that for any $x\in W$ and $X\in\mathcal{P}(W)$, $R^{\ni}Xx$ iff $x\in X$;
\item $R^{N}\subseteq W\times\mathcal{P}(W)$ such that for any $x\in W$ and $X\in\mathcal{P}(W)$, $R^{N}xX$ iff $X\in N(x)$.
\end{enumerate}

In this kind of semantic structures, we can define the following two-sorted normal bimodal language:

$$\phi::=p \mid \bot \mid \top \mid \neg\phi \mid \phi\land\phi \mid \phi\lor\phi \mid \Diamond_{N}\theta$$
$$\theta::=\Diamond_{\ni}\phi \mid \neg\theta \mid \theta\land\theta \mid \theta\lor\theta $$
where $\phi$ is a formula of the \emph{world type} and will be interpreted in the first domain, and $\theta$ is a formula of the \emph{subset type} and will be interpreted in the second domain. We can also define $\Box_{\ni}$ and $\Box_{N}$ in the standard way.

Given a two-sorted Kripke frame $\mathbb{F}^{2}=(W, \mathcal{P}(W), R^{\ni}, R^{N})$, a valuation $V$ is defined as a map $V:\mathsf{Prop}\rightarrow\mathcal{P}(W)$, where propositional variables are interpreted as subsets of the first domain. The satisfaction relation $\Vdash$ is defined as follows, for any $w\in W$ and any $X$ in $\mathcal{P}(W)$ (here we omit the Boolean connectives):

\begin{itemize}
\item $\mathbb{F}^{2},V,w\Vdash p$ iff $w\in V(p)$;
\item $\mathbb{F}^{2},V,w\Vdash \Diamond_{N}\theta$ iff there is an $X\in\mathcal{P}(W)$ such that $R^{N}wX$ and $\mathbb{F}^{2},V,X\Vdash\theta$;
\item $\mathbb{F}^{2},V,X\Vdash \Diamond_{\ni}\phi$ iff there is a $w\in W$ such that $R^{\ni}Xw$ and $\mathbb{F}^{2},V,w\Vdash\phi$.
\end{itemize}

\subsection{Bimodal translation}\label{Subsec:Bimodal:translation}
Now we are ready to define the translation $\tau$ from the INL language to the two-sorted normal bimodal language:
\begin{definition}(Bimodal translation)
Given any INL formula $\phi$, the bimodal translation $\tau(\phi)$ is defined as follows:
\begin{itemize}
\item $\tau(p)=p$;
\item $\tau(\bot)=\bot$;
\item $\tau(\top)=\top$;
\item $\tau(\neg\phi)=\neg\tau(\phi)$;
\item $\tau(\phi_1\land\phi_{2})=\tau(\phi_1)\land\tau(\phi_{2})$;
\item $\tau(\phi_1\lor\phi_{2})=\tau(\phi_1)\lor\tau(\phi_{2})$;
\item $\tau(\phi_1\to\phi_{2})=\tau(\phi_1)\to\tau(\phi_{2})$;
\item $\tau(\Box_{n}(\phi_1, \ldots, \phi_n; \phi))=\Diamond_{N}(\Diamond_{\ni}\tau(\phi_1)\land\ldots\land\Diamond_{\ni}\tau(\phi_n)\land \Box_{\ni}\tau(\phi))$.
\end{itemize}
\end{definition}

It is easy to see the following correctness result:

\begin{theorem}

For any neighbourhood frame $\mathbb{F}=(W,N)$, any valuation $V$ on $\mathbb{F}$, any $w\in W$, any INL formula $\phi$,
$$(\mathbb{F},V,w)\Vdash\phi\mbox{ iff }\mathbb{F}^{2},V, w\Vdash \tau(\phi).$$

\end{theorem}

\subsection{Sahlqvist correspondence theorem via bimodal translation}\label{Subsec:Proof2}

Similar to the normal modal logic case, we can define the Sahlqvist antecedents in the normal bimodal logic built up by boxed atoms and negative formulas in the inner part generated by $\land$, $\lor$, $\Diamond_{\ni}$, $\Diamond_{N}$, where the formulas are of the right type. Now we can prove Sahlqvist correspondence theorem by using bimodal translation:

\begin{theorem}\label{Thm:Sahl:bimodal}
For any INL formula $\phi\to\psi$, if $\phi$ is an INL-Sahlqvist antecedent and $\psi$ is a positive INL formula, then $\tau(\phi\to\psi)$ is a Sahlqvist formula in the normal bimodal language.
\end{theorem}

\begin{proof}
As we know, the Backus-Naur form of an INL-Sahlqvist antecedent is given as follows:
$$\zeta::=p \mid \bot \mid \top \mid \zeta\land\zeta \mid \nabla_{1,\theta}(\zeta)$$
$$\phi::=\zeta\mid \gamma\mid\bot\mid \top\mid \phi\land\phi \mid\phi\lor\phi \mid \Delta_{n,\theta}(\phi_1, \ldots, \phi_n)\mid \Box_{n}(\phi_1, \ldots, \phi_n;\zeta)\mid \Box_{n}(\phi_1, \ldots, \phi_n;\gamma),$$

where $\theta$ is a pure INL formula without propositional variables, $\zeta$ is a pseudo-boxed atom, and $\gamma$ is a negative formula. Therefore, the bimodal translations of $\tau(\zeta)$ and $\tau(\phi)$ have the following Backus-Naur form:

$$\tau(\zeta)::=p \mid \bot \mid \top \mid \tau(\zeta)\land\tau(\zeta) \mid\neg\Diamond_{N}(\Diamond_{\ni}\neg\tau(\zeta)\land \Box_{\ni}\tau(\theta))$$

$$\tau(\phi)::=\tau(\zeta)\mid \tau(\gamma)\mid\bot\mid \top\mid \phi\land\phi \mid\phi\lor\phi \mid $$
$$\Diamond_{N}(\Diamond_{\ni}\tau(\phi_1)\land\ldots\land\Diamond_{\ni}\tau(\phi_n)\land \Box_{\ni}\tau(\theta)) \mid$$
$$\Diamond_{N}(\Diamond_{\ni}\tau(\phi_1)\land\ldots\land\Diamond_{\ni}\tau(\phi_n)\land \Box_{\ni}\tau(\zeta)) \mid$$
$$\Diamond_{N}(\Diamond_{\ni}\tau(\phi_1)\land\ldots\land\Diamond_{\ni}\tau(\phi_n)\land \Box_{\ni}\tau(\gamma))$$

Now we analyze the shape of the Backus-Naur form above. For the bimodal translation of a pseudo-boxed atom $\zeta$ in the INL language, $\neg\Diamond_{N}(\Diamond_{\ni}\neg\tau(\zeta)\land \Box_{\ni}\tau(\theta))$ is equivalent to $\Box_{N}(\Box_{\ni}\tau(\zeta)\lor\neg \Box_{\ni}\tau(\theta))$. since $\theta$ is a pure formula without propositional variables, $\tau(\zeta)$ can be treated as a conjunction of boxed atoms in the bimodal language.

Now we examine $\tau(\phi)$. The Backus-Naur form of $\tau(\phi)$ is built up by $\tau(\zeta)$ (a conjunction of boxed atoms) and $\tau(\gamma)$ (a negative formula), generated by $\land,\lor$ and the three special shapes of $\tau(\Box_n(\phi_1, \ldots, \phi_n;\phi))$ where $\phi$ are pure formulas without propositional variables (the $\theta$ case), pseudo-boxed atoms (the $\zeta$ case) or negative formulas (the $\gamma$ case). It is easy to see that $\tau(\phi)$ is built up by pure formulas\footnote{Indeed, pure formulas are both negative and positive formulas in every propositional variable $p$, since their values are constants and $p$ does not occur in them.}, boxed atoms and negative formulas in the bimodal language, generated by $\Diamond_{\ni}, \Diamond_{N}, \land, \lor$, thus of the shape of Sahlqvist antecedent in the bimodal language. Therefore, $\tau(\phi\to\psi)$ is a Sahlqvist formula in the normal bimodal language.
\end{proof}

\section{Discussions and further directions}\label{Sec:Concl}

In this paper, we give two different proofs of Sahlqvist correspondence theorem, the first one by standard translation and minimal valuation, and the second one by reduction using the bimodal translation into a normal bimodal language. We give some remarks and further directions here.

\paragraph{Algebraic correspondence method using the algorithm ALBA.}
In \cite{CoGhPa14}, Sahlqvist and inductive formulas (an extension of Sahlqvist formulas, see \cite{GorankoV06} for further details) are defined based on duality-theoretic and order-algebraic insights. The Ackermann lemma based algorithm ALBA is given, which effectively computes first-order correspondents of input formulas/inequalities, and succeed on the Sahlqvist and inductive formulas/inequalities. In this approach, Sahlqvist and inductive formulas are defined in terms of the order-theoretic properties of the algebraic interpretations of the logical connectives. Indeed, in the dual complex algebra $\mathbb{A}$ of Kripke frame, the good properties of the connectives are the following:

\begin{itemize}
\item Unary $\Diamond$ is interpreted as a map $\Diamond^{\mathbb{A}}:\mathbb{A}\to\mathbb{A}$, which preserves arbitrary joins, i.e.\ $\Diamond^{\mathbb{A}}(\bigvee a)=\bigvee \Diamond^{\mathbb{A}}a$ and $\Diamond^{\mathbb{A}}\bot=\bot$. Similarly, $n$-ary diamonds are interpreted as maps which preserve arbitrary joins in every coordinate.
\item Unary $\Box$ is interpreted as a map $\Box^{\mathbb{A}}:\mathbb{A}\to\mathbb{A}$, which preserves arbitrary meets, i.e.\ $\Box^{\mathbb{A}}(\bigwedge a)=\bigwedge \Box^{\mathbb{A}}a$ and $\Box^{\mathbb{A}}\top=\top$. Preserving arbitrary meets guarantees the map $\Box^{\mathbb{A}}:\mathbb{A}\to\mathbb{A}$ to have a left adjoint $\Diamondblack^{\mathbb{A}}:\mathbb{A}\to\mathbb{A}$ such that $\Diamondblack^{\mathbb{A}}a\leq b\mbox{ iff }a\leq\Box^{\mathbb{A}} b$.
\end{itemize}

As we have seen from page \pageref{Page:preserv:join}, the algebraic interpretation of $\Box_{n}(\phi_{1}, \ldots, \phi_{n};\phi)$ preserves arbitrary joins in the first $n$ coordinates, and is monotone in the last coordinate. Therefore, we can adapt the ALBA method to the instantial neighbourhood logic case. In addition to this, we can also define INL-inductive formulas based on the algebraic properties of the instantial neighbourhood connectives, to extend INL-Sahlqvist formulas to INL-inductive formulas as well as to the language of instantial neighbourhood logic with fixpoint operators.

\paragraph{Completeness and canonicity.} Other issues that we do not study in the present paper include completeness of logics axiomatized by INL-Sahlqvist formulas and canonicity. For the proof of completeness, we need to establish the validity of INL-Sahlqvist formulas on their corresponding canonical frames, where canonicity and persistence might play a role (see \cite[Chapter 5]{BRV01}).

\paragraph{Acknowledgement} The author was supported in part by the Taishan Young Scholars Program of the Government of Shandong Province, China (No.tsqn201909151).

\bibliographystyle{abbrv}
\bibliography{INL}

\begin{thebibliography}{10}

\bibitem{BeEndG20}
N.~Bezhanishvili, S.~Enqvist, and J.~de~Groot.
\newblock Duality for instantial neighbourhood logic via coalgebra.
\newblock {\em ILLC Prepublication, PP-2020-08}, 2020.

\bibitem{BRV01}
P.~Blackburn, M.~de~Rijke, and Y.~Venema.
\newblock {\em Modal logic}, volume~53 of {\em Cambridge Tracts in Theoretical
  Computer Science}.
\newblock Cambridge University Press, 2001.

\bibitem{CoGhPa14}
W.~Conradie, S.~Ghilardi, and A.~Palmigiano.
\newblock Unified correspondence.
\newblock In A.~Baltag and S.~Smets, editors, {\em Johan van Benthem on Logic
  and Information Dynamics}, volume~5 of {\em Outstanding Contributions to
  Logic}, pages 933--975. Springer International Publishing, 2014.

\bibitem{GeNaVe05}
M.~Gehrke, H.~Nagahashi, and Y.~Venema.
\newblock A {S}ahlqvist theorem for distributive modal logic.
\newblock {\em Annals of Pure and Applied Logic}, 131(1-3):65--102, 2005.

\bibitem{GorankoV06}
V.~Goranko and D.~Vakarelov.
\newblock Elementary canonical formulae: Extending {S}ahlqvist's theorem.
\newblock {\em Annals of Pure and Applied Logic}, 141(1-2):180--217, 2006.

\bibitem{Ha03}
H.~H. Hansen.
\newblock Monotonic modal logics.
\newblock Master's thesis, Universiteit van Amsterdam, 2003.

\bibitem{KrWo99}
M.~Kracht and F.~Wolter.
\newblock Normal monomodal logics can simulate all others.
\newblock {\em Journal of Symbolic Logic}, 64:99--138, 1999.

\bibitem{Pa85}
R.~Parikh.
\newblock The logic of games and its applications.
\newblock In {\em Annals of Discrete Mathematics}, pages 111--140. Elsevier,
  1985.

\bibitem{Tu16}
O.~Tuyt.
\newblock Canonical rules on neighbourhood frames.
\newblock Master's thesis, University of Amsterdam, ILLC, Netherlands, 2016.

\bibitem{vB03}
J.~van Benthem.
\newblock Logic games are complete for game logics.
\newblock {\em Studia Logica}, 75(2):183--203, 2003.

\bibitem{vBBeEn19a}
J.~van Benthem, N.~Bezhanisivili, and S.~Enqvist.
\newblock A new game equivalence, its logic and algebra.
\newblock {\em Journal of Philosophical Logic}, 48(4):649--684, 2019.

\bibitem{vBBeEn19b}
J.~van Benthem, N.~Bezhanisivili, and S.~Enqvist.
\newblock A propositional dynamic logic for instantial neighbourhood semantics.
\newblock {\em Studia Logica}, 107(4):719--751, 2019.

\bibitem{vBBeEnYu17}
J.~van Benthem, N.~Bezhanisivili, S.~Enqvist, and J.~Yu.
\newblock Instantial neighbourhood logic.
\newblock {\em The Review of Symbolic Logic}, 10(1):116–144, 2017.

\bibitem{Yu18}
J.~Yu.
\newblock A tableau system for instantial neighborhood logic.
\newblock {\em Logical Foundations of Computer Science - International
  Symposium, LFCS 2018, Deerfield Beach, FL, USA, January 8-11, 2018,
  Proceedings}, pages 337--353, 2018.

\bibitem{Yu20}
J.~Yu.
\newblock Lyndon interpolation theorem of instantial neighborhood logic -
  constructively via a sequent calculus.
\newblock {\em Annals of Pure and Applied Logic}, 171(1), 2020.

\end{thebibliography}
\end{document}